\documentclass{article}

\usepackage{amssymb}
\usepackage{amsmath}
\usepackage{amsthm}
\usepackage{a4wide}

\usepackage[english]{babel}

\newcommand{\cenlin}{\centerline}

\newcommand{\ovline}{\overline}
\newcommand{\wtilde}{\widetilde}

\newcommand{\vect}[1]{\mathbf{#1}}

\newcommand{\larr}{\leftarrow}
\newcommand{\rarr}{\rightarrow}
\newcommand{\xrarr}[2]{\xrightarrow[#1]{#2}}

\newcommand{\Rarr}{\Rightarrow}
\newcommand{\LRarr}{\Leftrightarrow}

\newcommand{\backsl}{\backslash}

\newcommand{\nullset}{\varnothing}

\newcommand{\lesq}{\leqslant}
\newcommand{\greq}{\geqslant}

\newcommand{\veps}{\varepsilon}

\newcommand{\forceparindent}{\hskip 1.5em}

\newcommand{\vgap}{\vskip 3mm}

\newcounter{Formulanum}
\newcommand{\equ}[1]{\refstepcounter{Formulanum}%
\label{#1}%
(\arabic{Formulanum})}
\newcommand{\equref}[1]{(\ref{#1})}

\newtheorem{corollary}{Corollary}
\newtheorem{definition}{Definition}
\newtheorem{example}{Example}
\newtheorem{proposition}{Proposition}

\newenvironment{exampledescript}{$\vartriangleleft$}{$\vartriangleright$}

\begin{document}

\title{On sound ranging in proper metric spaces}

\author{Sergij V. Goncharov\thanks{Faculty of Mechanics and Mathematics, Oles Honchar Dnipro National University, 72 Gagarin Avenue, 49010 Dnipro, Ukraine.
\textit{E-mail: goncharov@mmf.dnulive.dp.ua}}}

\date{August 2018}

\maketitle

\begin{abstract}
We consider the sound ranging, or source localization, problem ---
find the source-point from the moments when the wave-sphere
of linearly, with time, increasing radius reaches the sensor-points ---
in the proper metric spaces (any closed ball is compact)
and, in particular, in the finite-dimensional normed spaces.
We approximate the solution to arbitrary precision by the iterative process
with the stopping criterion.
\end{abstract}

\cenlin{\small \textit{MSC2010:} Pri 41A65, Sec 54E50, 46B20, 40A05, 68W25}

\cenlin{\small \textit{Keywords:} sound ranging, localization, approximation, algorithm,
proper metric space, normed space}

\section*{Introduction}

\forceparindent
Let $(X;\rho)$ be a metric space, i.e. the set $X$ with the metric $\rho \colon X\times X \rarr \mathbb{R}_+$.
Let $\vect{s} \in X$ be an unknown point, ``source''.
At unknown moment $t_0 \in \mathbb{R}$ of time the source ``emits the (sound) wave'', which is the sphere
$\bigl\{ \vect{x} \in X \mid \rho(\vect{x}; \vect{s}) = v(t - t_0) \bigr\}$
for any moment $t \greq t_0$. We assume, without loss of generality, that ``sound velocity'' $v = 1$
(switch to scaled time $t \larr vt$ if $v \ne 1$).

Let $\{ \vect{r}_i \}_{i \in I}$, $\vect{r}_i \in X$, be an indexed set of known ``sensors''.
For each sensor we know the moment $t_i$ when it was reached by the expanding wave;
that is, $t_i = t_0 + \rho (\vect{r}_i; \vect{s})$ are known.

The \textit{sound ranging problem} (SRP), also called \textit{source localization},
is to find $\vect{s}$ and $t_0$ from known moments when the wave reaches known sensors,
$(\{ \vect{r}_i \}; \{ t_i \})$.

SRPs of this and more general forms, usually in Euclidean space, appear in acoustics, geophysics, navigation,
sensor networks, tracking among the others;
there is an abundant literature of the subject and of the proposed techniques,
see e.g. \cite[1]{compagnoni2017}, \cite[9.1]{huang2004} for further references.

\vgap

In \cite{goncharov2018} we investigated noiseless SRP in the infinite-dimensional separable Hilbert space $H$.
The method there is, basically, the ``classical'' one applied in $\mathbb{R}^m_2$, ---
we solve the set of implied equations $(t_i - t_0)^2 = \rho^2 ( \vect{r}_i ; \vect{s} ) = \sum\limits_j (r^{(i)}_j - s_j)^2$,
where $t_0$ and coordinates $\{ s_j \}_{j\in \mathbb{N}}$ of $\vect{s}$ are unknowns,
--- with few technicalities related to the countability of coordinates.
It is a method of ``solving'' kind in that we express exact values of $s_j$
through known parameters $t_i$ and $r^{(i)}_j$, in closed form.

This time we look into another generalization of SRP, without Euclidicity in general case.
The classical approach doesn't work anymore, because the coordinates, if there are any,
are not so easily ``extractable'' from the equations $t_i - t_0 = \rho( \vect{r}_i ; \vect{s} )$,
which become significantly nonlinear.

Instead, we describe the more or less ``universal'' iterative process that ``converges'' to the source
in certain sense explained further; this is a method of ``approximating'' kind.
In short, we cover the regions of the space by the balls, and repeatedly refine the cover by
a) replacing every ball with its cover by the balls of halved radius,
then b) removing from the cover each ball such that certain ``deviation''
at its center is greater than its radius. ``Deviation'' at the source is 0.

\vgap

It is presented as an algorithm.
How practical such algorithm is depends on its ``executor'',
or, in other words, what (and how many) elementary actions the executor is allowed to perform.
If e.g. it were a ``computer'' $\mathbb{U}$ with $\mathrm{card}(X)$ cores,
we would plainly assign to each core single $\vect{x} \in X$
to verify if $\vect{x}$ is a solution ($t_i - \rho(\vect{r}_i; \vect{x}) \equiv \mathrm{const}$, see below).
The target executor for our algorithm is far below $\mathbb{U}$,
it is closer to the ``general purpose'' computing devices of nowadays.

\vgap

\textbf{Disclaimer.}
The intent of this paper is not to proclaim the ``novelty'' of the method being described
(to put it mildly, that would be dubious),
but to ``plant'' (develop) the ``essence'' (approach) of akin methods
in more general ``soil'' (context) and watch how it ``blossoms'' (works).

We consider ``empty'' spaces without ``physics'' such as echoes, varying sound velocity,
noisy measurements (except for one remark),
focusing on rigour rather than realness and applicability.

Similarly to the $H$ case, the content of this paper has a ``folklore'' flavor, so ---

\vgap

\textbf{Acknowledgements.}
We thank everyone who 1) points out where these results
or their generalizations have been obtained already
(some paper from 1920--30s? something like \cite{kedlaya2002}?), or
2) by means of a time machine, delivers this paper to 1920--30s, when it should've appeared\textellipsis

\vgap

\textbf{Root finding vs. Optimization.} Searching for the minimum of $f(\vect{x})$
(or the maximum of $-f(\vect{x})$) is the \textit{optimization} problem
that is part of the most of approaches to solving SRPs, especially with noised measurements
($f(\vect{x})$ is called ``cost'' or ``plausibility'' function there).
It is performed either (1-stage) directly in the space of possible source positions
to estimate the actual position, or (2-stage) in the space of relative time-delays
$t_i - t_j$ between sensors to estimate these delays,
which then allow to obtain the source position in closed form or,
alternatively, estimate it as well;
see \cite{alamedapineda2014}, \cite{bestagini2013}, \cite{chen2002}, \cite{gillette2008}, \cite{huang2004}.
For example, in \cite{alamedapineda2014} the branch \& bound technique is applied
and compared to other ones.
The maximum likelihood estimator is one of common approaches to such optimization as well,
though there are issues with local minima when the cost function isn't strictly concave
(\cite[9.4]{huang2004}).
The Euclidicity of the space where the wave propagates is important
in deriving the closed form solutions and in the least squares localizations
(\cite[4]{bestagini2013}, \cite[9.5]{huang2004}).

In our simplified case the exact delays are known
and the non-negative function has unique zero; we search for that zero, rather than the extremum,
in the (non-Euclidean) space of possible source positions.
This is a \textit{root finding} of ``bracketing'', or ``exclude \& enclose'', type
(see \cite{byrne2008}, \cite{ostrowski1973}).

\vgap

The bibliography with somewhat more emphasis on the practice of sound ranging,
including historical surveys, was given in \cite{goncharov2018};
or better, see \cite{aubin2014}, \cite{compagnoni2017}, \cite[9]{huang2004}.

\vgap

``$\clubsuit$'' indicates the assumptions, or constraints,
that we require to hold unless stated otherwise.
``$\bullet$'' is for the statements that are considered to be well-known under given assumptions
(see e.g. \cite{giles1987}, \cite{jameson1974}, \cite{kolmogorov1975},
\cite{montesinos2015}, \cite{searcoid2007}) and included for the sake of completeness, without proofs or references.

\section{SR in proper metric spaces}

\subsection{Preliminaries}

\forceparindent
$\bullet$ ``2nd $\bigtriangleup$-inequality'':
$\forall \vect{x}, \vect{y}, \vect{z} \in X$
$\bigl| \rho(\vect{x}; \vect{z}) - \rho(\vect{z}; \vect{y}) \bigr| \lesq \rho(\vect{x}; \vect{y})$.

As usual, $\vect{x}_k \xrarr{k \rarr \infty}{} \vect{y}$ means $\rho(\vect{x}_k; \vect{y}) \xrarr{k \rarr \infty}{} 0$.

$\bullet$ Continuity of metric: $\vect{x}_k \xrarr{k \rarr \infty}{} \vect{y}$ $\Rarr$
$\rho(\vect{x}_k; \vect{z}) \xrarr{k \rarr \infty}{} \rho(\vect{y}; \vect{z})$.

$B(\vect{c}; r) = \{ \vect{x} \in X \mid \rho(\vect{x}; \vect{c}) < r \}$
and $B[\vect{c}; r] = \{ \vect{x} \in X \mid \rho(\vect{x}; \vect{c}) \lesq r \}$
denote the open and closed \textit{balls} with center $\vect{c}$ and of radius $r$.

$\bullet$ For any $B[\vect{c}; r]$ and any point $\vect{a}$, $\rho(\vect{a}; \vect{c}) = d$, we have

\cenlin{$\forall \vect{x} \in B[\vect{c}; r]$: $d - r \lesq \rho(\vect{x}; \vect{a}) \lesq d + r$}

The set $A \subseteq X$ is said to be \textit{compact} if $\forall \{ \vect{x}_k \}_{k\in \mathbb{N}} \subseteq A$:
$\exists$ $\{ \vect{x}_{k_l} \}_{l \in \mathbb{N}}$: $\vect{x}_{k_l} \xrarr{l \rarr \infty}{} \vect{x}_0 \in A$.

$\bullet$ If $A$ is compact, then any closed subset of $A$ is compact too.

The family of sets $\{ C_j \}_{j \in J}$, $C_j \subseteq X$, is said to be a \textit{cover} of $A \subseteq X$
if $A \subseteq \bigcup\limits_{j \in J} C_j$.

$\bullet$ The closed $A \subseteq X$ is compact if and only if any open cover of $A$ has finite subcover.

The set $A \subseteq X$ is called \textit{bounded} if
$\mathrm{diam}\, A = \sup\limits_{\vect{x}, \vect{y} \in A} \rho(\vect{x}; \vect{y}) < \infty$.

{\ }

$\clubsuit$M1. $(X; \rho)$ is \textit{proper}: any closed ball is compact.

Such spaces are also called \textit{finitely compact} or
having the \textit{Heine-Borel/Bolzano-Weierstrass property}
(\cite[1.5, p. 43]{deza2009}, \cite[1.4, p. 32]{papadopoulos2014};
in addition, see \cite{williamson1987}).

{\ }

$\bullet$ In this definition, ``any closed ball'' can be replaced with ``any closed, bounded subset''.

$\bullet$ A proper metric space is complete: any fundamental sequence converges.

In fact, it would suffice that
$\exists \delta > 0$: $\forall \vect{x} \in X$ $B[\vect{x}; \delta]$ is compact. Indeed,
if $\{ \vect{x}_k \}_{k=1}^{\infty}$ is fundamental, then
$\exists N$: $\{ \vect{x}_k \}_{k=N}^{\infty} \subseteq B[\vect{x}_N; \delta] = B$ $\Rarr$
$\exists \{ \vect{x}_{N+k_l} \}_{l=1}^{\infty}$: $\vect{x}_{N+k_l} \xrarr{l \rarr \infty}{} \vect{y} \in B$,
implying $\vect{x}_k \xrarr{k \rarr \infty}{} \vect{y}$ as well.
Also, the converse fails: infinite-dimensional $H$ is complete, but not proper.

{\ }

Now we proceed to the SRP.
The source $\vect{s} \in X$ and the emission moment $t_0 \in \mathbb{R}$ are unknown.

{\ }

$\clubsuit$1. The set of sensors is finite: $\{ \vect{r}_i \}_{i=1}^n$, $\vect{r}_i \in X$, and
$\vect{r}_i \ne \vect{r}_j$, $i \ne j$.

{\ }

These sensors and the moments

\hfill $t_i = t_0 + \rho(\vect{r}_i; \vect{s})$, $i = \ovline{1, n}$ \hfill \equ{eqMoments}

\noindent
define the SRP
$\bigl( \{ \vect{r}_i \} ; \{ t_i \} \bigr)$; each pair $(\vect{s'}; t')$ satisfying the set of equations

\hfill $t_i = t' + \rho( \vect{r}_i; \vect{s}' )$, $i = \ovline{1, n}$ \hfill \equ{eqMain}

\noindent
is a \textit{solution} of this SRP.
Since $t'$ is defined uniquely from any such equation for given $\vect{s}'$,
the source $\vect{s}'$ itself can be called a solution too.

{\ }

$\clubsuit$2. The solution $\vect{s}$ of the SRP $\bigl( \{ \vect{r}_i \} ; \{ t_i \} \bigr)$ is unique.

{\ }

Obviously, this is not the general case. For $n=1$ any $\vect{y} \in X$ is a solution,
with $t' = t_1 - \rho( \vect{y}; \vect{r}_1 )$.
In $\mathbb{R}^2_2$ we can place ``true'' and ``false'' sources, $\vect{s}$ and $\vect{s}'$
respectively, at 2 foci of hyperbola, and place 3 sensors on the same branch of that hyperbola.
Then $\rho( \vect{r}_i ; \vect{s} ) - \rho( \vect{r}_i; \vect{s}' ) \equiv d$,
thus $\vect{s}'$ emitting the wave at the moment $t' = t_0 + d$ is another solution.

To ensure the uniqueness of the solution in $\mathbb{R}^m_2$,
we can take $m+2$ sensors such that $\{ \vect{r}_2 - \vect{r}_1$; \ldots; $\vect{r}_{m+1} - \vect{r}_1 \}$
is a basis of $\mathbb{R}^m$ and $\vect{r}_{m+2} = 2 \vect{r}_1 - \vect{r}_2$
(\cite[Prop. 4]{goncharov2018}).

\begin{definition}
For any $\vect{x} \in X$ the \emph{backward moments}

\cenlin{$\tau_i(\vect{x}) := t_i - \rho( \vect{x}; \vect{r}_i )$, $i=\ovline{1,n}$}
\end{definition}

$\tau_i(\vect{x})$ is the moment when the wave must be emitted from $\vect{x}$ to reach $\vect{r}_i$ at the moment $t_i$.

\begin{definition}\label{defDefect}
For any $\vect{x} \in X$ the \emph{defect}

\cenlin{$D(\vect{x}) := \frac{1}{n} \sum\limits_{i=1}^n
\bigl| \tau_i(\vect{x}) - \frac{1}{n} \sum\limits_{j=1}^n \tau_j(\vect{x}) \bigr|$}
\end{definition}

Cf. e.g. \cite[3]{bestagini2013} or \cite[2.4]{pollefeys2008}.
We rewrite

\cenlin{$D(\vect{x}) = \frac{1}{n^2} \sum\limits_{i=1}^n
\bigl| n \tau_i(\vect{x}) - \sum\limits_{j=1}^n \tau_j(\vect{x}) \bigr| =
\frac{1}{n^2} \sum\limits_{i=1}^n
\bigl| \sum\limits_{j=1}^n \bigl[ \tau_i(\vect{x}) - \tau_j(\vect{x}) \bigr] \, \bigr| \lesq
\frac{1}{n^2} \sum\limits_{i=1}^n
\sum\limits_{j=1}^n \bigl| \tau_i(\vect{x}) - \tau_j(\vect{x}) \bigr|$}

The elementary properties of $D(\cdot)$ follow
(Props. \ref{propDefectZeroSol}--\ref{propConcentrationNearZero}).

\begin{proposition}\label{propDefectZeroSol}
$\vect{s}' \in X$ is the solution of the SRP if and only if $D(\vect{s}') = 0$.
\end{proposition}

\begin{proof}
If $\vect{s}'$ is such solution, then $t_i = t' + \rho( \vect{r}_i ; \vect{s}' )$,
$\tau_i(\vect{s}') \equiv t'$ $\Rarr$ $\tau_i(\vect{s}') - \tau_j(\vect{s}') \equiv 0$, so $D(\vect{s}') = 0$.
Contrariwise, $D(\vect{s}') = 0$ implies $\tau_i(\vect{s}') \equiv t' = \frac{1}{n} \sum\limits_{j=1}^n \tau_j(\vect{s}')$,
and $(\vect{s}'; t')$ is the solution.
\end{proof}

\begin{corollary}\label{corDefectSingleZero}
$D(\vect{x})$ has exactly one zero in $X$, at $\vect{x} = \vect{s}$.
\end{corollary}

\begin{proposition}\label{propDefectTwoDiffsNotLess}
$\forall \vect{x} \in X$ $\exists i, j$: $|\tau_i(\vect{x}) - \tau_j(\vect{x})| \greq D(\vect{x})$.
\end{proposition}

\begin{proof}
Assuming the contrary, we have
$D(\vect{x}) < \frac{1}{n^2} \sum\limits_{i=1}^n \sum\limits_{j=1}^n D(\vect{x}) = D(\vect{x})$ ---
a contradiction.
\end{proof}

\begin{proposition}\label{propDefectDiffUppEstim}
$\forall \vect{x}, \vect{y} \in X$:
$\bigl| D(\vect{x}) - D(\vect{y}) \bigr| \lesq 2 \rho(\vect{x}; \vect{y})$.
\end{proposition}

\begin{proof}
$\bigl| D(\vect{x}) - D(\vect{y}) \bigr| =
\frac{1}{n^2} \bigl| \sum\limits_i | \ldots | - \sum\limits_i | \ldots | \, \bigr| =
\frac{1}{n^2} \bigl| \sum\limits_i \bigl[ \, | \sum\limits_j (\ldots) | - | \sum\limits_j (\ldots) | \, \bigr] \, \bigr| \lesq$

\cenlin{$\lesq \frac{1}{n^2} \sum\limits_i \bigl| \sum\limits_j (\ldots) - \sum\limits_j (\ldots) \bigr| =
\frac{1}{n^2} \sum\limits_{i=1}^n \Bigl| \sum\limits_{j=1}^n
\bigl[ \tau_i(\vect{x}) - \tau_j(\vect{x}) -
\{ \tau_i(\vect{y}) - \tau_j(\vect{y}) \} \bigr] \, \Bigr| \lesq$}

\cenlin{$\lesq \frac{1}{n^2} \sum\limits_{i=1}^n \sum\limits_{j=1}^n \Bigl| \bigl[ \tau_i(\vect{x}) -
\tau_i(\vect{y}) \bigr] - \bigl[ \tau_j(\vect{x}) - \tau_j(\vect{y}) \bigr] \, \Bigr| \lesq
\frac{1}{n^2} \sum\limits_{i,j} \Bigl[ \, \bigl| \tau_i(\vect{x}) -
\tau_i(\vect{y}) \bigr| + \bigl| \tau_j(\vect{x}) - \tau_j(\vect{y}) \bigr| \, \Bigr] =$}

\hfill $= \frac{1}{n^2} \sum\limits_{i,j} \bigl[ \, \bigl| \rho( \vect{x}; \vect{r}_i ) -
\rho( \vect{y}; \vect{r}_i ) \bigr| + \bigl| \rho( \vect{x}; \vect{r}_j ) -
\rho( \vect{y}; \vect{r}_j ) \bigr| \, \bigr] \lesq
\frac{1}{n^2} \sum\limits_{i,j} \bigl[ \rho( \vect{x}; \vect{y} ) + \rho( \vect{x}; \vect{y} ) \bigr] =
2 \rho(\vect{x}; \vect{y})$ \hfill
\end{proof}

\noindent
--- $D(\cdot)$ is a Lipschitz function (see \cite[6]{heinonen2001}, \cite[9.4]{searcoid2007}).

\begin{corollary}\label{corDefectUnifCont}
$D(\vect{x})$ is uniformly continuous on $X$.
\end{corollary}

\begin{proposition}\label{propConcentrationNearZero}
If $\vect{s} \in B = B[\vect{c}; r]$, then
$\forall \delta > 0$ $\exists \veps > 0$: $\vect{x} \in B$, $D(\vect{x}) < \veps$ $\Rarr$
$\rho( \vect{x} ; \vect{s}) < \delta$.
\end{proposition}

\begin{proof}
Let $S = \bigl\{ \vect{x} \in B \mid \rho(\vect{x}; \vect{s}) \greq \delta \bigr\}$
and $\veps = \inf\limits_{\vect{x} \in S} D(\vect{x})$.
We claim that $S$ is compact: indeed, $S \subset B$ and $S$ is closed due to continuity of metric.
Since $\forall \vect{x} \in S$: $D(\vect{x}) > 0$ due to Cor. \ref{corDefectSingleZero},
we have $\veps > 0$ (otherwise $\forall k \in \mathbb{N}$ $\exists \vect{x}_k \in S$: $D(\vect{x}_k) < \frac{1}{k}$,
and it follows from compactness of $S$ that
$\exists \{ \vect{x}_{k_l} \}_{l\in \mathbb{N}}$:
$\vect{x}_{k_l} \xrarr{l \rarr \infty}{} \vect{x}_0 \in S$;
$D(\vect{x})$ is continuous, so $D(\vect{x}_{k_l}) \xrarr{l \rarr \infty}{} D(\vect{x}_0)$,
but $0 \lesq D(\vect{x}_{k_l}) < \frac{1}{k_l} \lesq \frac{1}{l}$ $\Rarr$
$D(\vect{x}_0) = \lim\limits_{l \rarr \infty} D(\vect{x}_{k_l}) = 0$ --- a contradiction).

Now, if $\vect{x} \in B$ and $D(\vect{x}) < \veps$, then $\vect{x} \notin S$,
which means $\rho(\vect{x}; \vect{s}) < \delta$.
\end{proof}

\textbf{Test for a ball.} Consider arbitrary closed ball $B = B[\vect{c}; r] \subseteq X$.
If $\vect{s} \in B$, then

\cenlin{$\rho(\vect{r}_i; \vect{c}) - r \lesq \rho(\vect{r}_i; \vect{s}) \lesq
\rho(\vect{r}_i; \vect{c}) + r$, $i = \ovline{1,n}$ $\LRarr$}

\cenlin{$\LRarr$ $t_0 = t_i - \rho(\vect{r}_i; \vect{s}) \in \bigl[ t_i - \rho(\vect{r}_i; \vect{c}) - r;
t_i - \rho(\vect{r}_i; \vect{c}) + r \bigr]$, $i = \ovline{1,n}$}

\noindent
hence $t_0 \in C = \bigcap\limits_{i=1}^n \bigl[ \tau_i(\vect{c}) - r;
\tau_i(\vect{c}) + r \bigr] \ne \nullset$.
It is easy to see that $C \ne \nullset$ if and only if

\cenlin{$2r \greq \max\limits_i \tau_i(\vect{c}) - \min\limits_i \tau_i(\vect{c}) =
\max\limits_{i, j} \bigl| \tau_i(\vect{c}) - \tau_j(\vect{c}) \bigr| =: I(\vect{c})$}

Thus we have the inference:
if $\vect{s} \in B[\vect{c}; r]$,
then $2r \greq I(\vect{c})$.
Conversely, $2r < I(\vect{c})$ $\Rarr$ $\vect{s} \notin B[\vect{c}; r]$.
From $|\tau_i(\vect{c}) - \tau_j(\vect{c})| \lesq I(\vect{c})$ it follows that
$D(\vect{c}) \lesq \frac{1}{n^2} \sum\limits_{1\lesq i,j \lesq n} I(\vect{c}) = I(\vect{c})$, so

\hfill $2r < D(\vect{c})$ $\Rarr$ $\vect{s} \notin B[\vect{c}; r]$ \hfill \equ{eqNegCond}

\noindent
and the condition ``$2r < D(\vect{c})$''
divides the family of all closed balls in $X$ into 2 families:

1) $\mathcal{N}$ (egative) --- the balls that satisfy this condition and thus do not contain $\vect{s}$,

2) $\mathcal{S}$ (uspicious) --- the balls that do not satisfy it.

Of course, even if $B \in \mathcal{S}$, more ``advanced'' tests may prove that $\vect{s} \notin B$.

\begin{proposition}\label{propNegBallSmallEnough}
Suppose $D(\vect{x}) > 0$.
If $\vect{x} \in B = B[\vect{y}; r]$ and $r < \frac{1}{4} D(\vect{x})$, then $B \in \mathcal{N}$.
\end{proposition}

\begin{proof}
By Prop. \ref{propDefectDiffUppEstim},
$\rho(\vect{y}; \vect{x}) \lesq r < \frac{1}{4} D(\vect{x})$ $\Rarr$
$|D(\vect{y}) - D(\vect{x})| < \frac{1}{2}D(\vect{x})$ $\Rarr$ $D(\vect{y}) > \frac{1}{2} D(\vect{x}) > 0$.

Since $r < \frac{1}{4} D(\vect{x}) < \frac{1}{2} D(\vect{y})$, \equref{eqNegCond} implies $B \in \mathcal{N}$.
\end{proof}

\noindent
--- if $D(\vect{x}) > 0$, then the suspicious balls that are small enough do not contain $\vect{x}$.

{\ }

\textbf{Refining cover.} Consider any $B = B[\vect{z}; r]$.
$\mathcal{A}_0 = \bigl\{ B(\vect{c}; \frac{r}{2}) \mid \vect{c} \in B \bigr\}$ is a cover of $B$ by open sets.
From $\clubsuit$M1 it follows that there exists a finite subcover
$\mathcal{A}_1 = \bigl\{ B(\vect{c}_i; \frac{r}{2}) \mid i=\ovline{1,N} \bigr\} \subseteq \mathcal{A}_0$,
which also covers $B$.
Clearly, $\mathcal{B} = \bigl\{ B[\vect{c}_i; \frac{r}{2}] \mid i=\ovline{1,N} \bigr\}$
is a finite cover of $B$ too. Thus we obtained

\begin{proposition}\label{propRefinCoverMetric}
For any $B = B[\vect{z}; r]$ there exists a finite cover
$\mathcal{B} = \mathcal{B}(B) = \bigl\{ B[\vect{c}_i; \frac{r}{2}] \mid i=\ovline{1,N} \bigr\}$ of $B$,
and $\forall B[\vect{c}; \frac{r}{2}] \in \mathcal{B}$: $\rho(\vect{c}; \vect{z}) \lesq r$.
\end{proposition}

Naturally, we want $N$ to be as small as possible;
however, the time spent in the (intricate) positioning of less balls
can exceed the time gained by not testing more balls.
This topic is omitted here; cf. \cite{dumer2007}, \cite[2]{goodman2018}.
If $\exists N_d \in \mathbb{N}$ such that any $B[\vect{z}; r]$ can be covered by at most $N_d$ closed balls
of radius $\frac{r}{2}$, then $(X; \rho)$ is called \textit{doubling},
and $N_d$ is its \textit{doubling constant} (\cite[10.13, p. 81]{heinonen2001}).


For the method at hand we can weaken $\rho(\vect{c}; \vect{z}) \lesq r$
to $\rho(\vect{c}; \vect{z}) \lesq K_1 r$, $K_1 \greq 1$
\hfill \equ{eqCoverDistCenters}

{\ }

When $(X;\rho)$ has some additional properties, e.g. its points can be provided with coordinates,
we are able to make $\mathcal{B}$ more ``constructively'';
in particular, next section describes this procedure in the finite-dimensional normed spaces.
Another example is the Riemannian manifolds with intrinsic metric
(see \cite[7.1, 8.1]{deza2009}, \cite[3.3]{petersen2006}),
optionally immersed in $\mathbb{R}^m_2$.

{\ }

Let $\mathcal{B}_0 = \{ B[\vect{z};r] \}$ and $\mathcal{B}_1 = \mathcal{B}(B[\vect{z}; r])$.
We then denote by $\mathcal{B}_2$ the union of the covers of all balls from $\mathcal{B}_1$, \textellipsis,
by $\mathcal{B}_k$ the union of the covers of all balls from $\mathcal{B}_{k-1}$, \textellipsis

At that, the balls in $\mathcal{B}_k$ are of radius $\frac{r}{2^k}$ and
$\bigcup\limits_{B \in \mathcal{B}_k} B \subseteq \bigcup\limits_{B \in \mathcal{B}_{k+1}} B$.

\begin{proposition}\label{propAllCoveringsInBall}
Let $\mathcal{B}_{\infty} = \bigcup\limits_{k=0}^{\infty} \mathcal{B}_k$.
Then $\bigcup\limits_{B \in \mathcal{B}_{\infty}} B \subseteq B[\vect{z}; Kr]$, where $K = 2K_1 + 1$.
\end{proposition}

\begin{proof}
By construction, $\forall B' = B[\vect{z}_k; \frac{r}{2^k}] \in \mathcal{B}_k$
$\exists B'' = B[\vect{z}_{k-1}; \frac{r}{2^{k-1}}] \in \mathcal{B}_{k-1}$:
$B' \in \mathcal{B}(B'')$, so by \equref{eqCoverDistCenters}:
$\rho(\vect{z}_k; \vect{z}_{k-1}) \lesq K_1 \cdot \frac{r}{2^{k-1}} = 2K_1 r \cdot \frac{1}{2^k}$.
Therefore

\cenlin{$\rho(\vect{z}_k; \vect{z}) \lesq \rho(\vect{z}_k; \vect{z}_{k-1}) + \rho(\vect{z}_{k-1}; \vect{z}) \lesq
2K_1 r \cdot \frac{1}{2^k} + \rho(\vect{z}_{k-1}; \vect{z}) \lesq$}

\cenlin{$\lesq 2K_1 r \cdot \frac{1}{2^k} + 2K_1 r \cdot \frac{1}{2^{k-1}} + \rho(\vect{z}_{k-2}; \vect{z}) \lesq \ldots \lesq
2K_1 r \sum\limits_{i=1}^k 2^{-i} + \rho(\vect{z}; \vect{z}) \lesq 2K_1 r$}

Now, $\forall \vect{x} \in \bigcup\limits_{B \in \mathcal{B}_{\infty}} B$:
$\exists B[\vect{z}_k; \frac{r}{2^k}] \ni \vect{x}$ for some $k$, hence

\hfill $\rho(\vect{x}; \vect{z}) \lesq \rho(\vect{x}; \vect{z}_k) + \rho(\vect{z}_k; \vect{z}) \lesq
\frac{r}{2^k} + 2K_1 r \lesq (2K_1 + 1) r$ \hfill
\end{proof}

\subsection{Method}

\forceparindent
We add one more assumption for the sake of simplicity:

{\ }

$\clubsuit$3. At least one ball $B_{0,1} = B[\vect{c}_{0,1}; r] \ni \vect{s}$ is known.

{\ }

That means the ball being ``big enough'' to contain all possible positions of the source.

{\ }

\textbf{Step 0.} Let $k = 0$, $\mathcal{C}_0 = \{ B_{0,1} \}$, and $r_0 = r$.

{\ }

\textbf{Step 1.} Let $\mathcal{C}_{k+1} = \nullset$.

For each ball $B = B[\vect{y}; r_k] \in \mathcal{C}_k$, $r_k = \frac{r}{2^k}$,
there is the cover $\mathcal{B}$ of $B$, which
consists of the balls $B' = B[\vect{c}; r_{k+1}]$, $r_{k+1} = \frac{1}{2} r_k = \frac{r}{2^{k+1}}$.

Consider each $B'$ in turn and apply test \equref{eqNegCond} to $B'$.
If $B' \in \mathcal{S}$, then add $B'$ to $\mathcal{C}_{k+1}$.

{\ }

\textbf{Step 2.} Let $\vect{z}_{k+1}$ be the center of the arbitrarily chosen ball from $\mathcal{C}_{k+1}$.

{\ }

\textbf{Step 3.} $k := k + 1$, goto Step 1.

{\ }

It follows from $\vect{s} \in B_{0, 1} \subseteq \bigcup\limits_{B \in \mathcal{B}(B_{0, 1})} B$
that at least one ball from $\mathcal{B}(B_{0,1})$ contains $\vect{s}$,
this ball belongs to $\mathcal{S}$ and therefore $\mathcal{C}_1 \ne \nullset$.
Similarly, at least one ball from $\bigcup\limits_{B \in \mathcal{C}_1} \mathcal{B}(B)$
contains $\vect{s}$, implying $\mathcal{C}_2 \ne \nullset$, etc.:
$\forall k\in \mathbb{Z}_+$ $\mathcal{C}_k \ne \nullset$.
Hence these steps define the infinite sequence of the covers $\{ \mathcal{C}_k \}_{k=0}^{\infty}$
and the infinite sequence of the centers $\{ \vect{z}_k \}_{k=1}^{\infty}$.

\begin{proposition}
$\vect{z}_k \xrarr{k \rarr \infty}{} \vect{s}$.
\end{proposition}

\begin{proof}
By Prop. \ref{propAllCoveringsInBall},
$\{ \vect{z}_k \}_{k \in \mathbb{N}} \subseteq \widehat{B} =  B[\vect{c}_{0,1}; Kr]$
and $\vect{s} \in B_{0,1} \subseteq \widehat{B}$.	

Take any $\delta > 0$. By Prop. \ref{propConcentrationNearZero},
$\exists \veps > 0$: $\vect{x} \in \widehat{B}$, $D(\vect{x}) < \veps$ $\Rarr$
$\rho(\vect{x}; \vect{s}) < \delta$.

Since $r_k = \frac{r}{2^k} \xrarr{k \rarr \infty}{} 0$, we see that $\exists k_0$:
$\forall k \greq k_0$ $r_k < \frac{1}{2} \veps$.
By construction, $\forall B = B[\vect{c}; r_k] \in \mathcal{C}_k$: $B \in \mathcal{S}$, so
$D(\vect{c}) \lesq 2 r_k < \veps$. Hence $\rho(\vect{c}; \vect{s}) < \delta$;
in particular, $\rho(\vect{z}_k; \vect{s}) < \delta$.
\end{proof}

In practice, however, we would like to know when to halt this process.
We need a discernible ``sign'' that $\vect{z}_k$ is close enough to $\vect{s}$,
$\rho(\vect{z}_k; \vect{s}) < \delta$ for the preselected precision $\delta$.
We do not rely on the condition $r_k < \frac{1}{2} \veps$, because $\veps$ is, in a sense,
unknown --- defined ``not constructively enough''; put differently, the convergence rate is unknown.

To attain this goal, we add the stopping criterion to Step 3 and replace it by

{\ }

\textbf{Step 3'.} $k := k + 1$.
If 1) $\rho(\vect{c}'; \vect{c''}) < \frac{2}{3} \delta$ for any
$B[\vect{c}'; r_k]$, $B[\vect{c}''; r_k] \in \mathcal{C}_k$
and 2) $r_k < \frac{1}{3} \delta$, then halt; else goto Step 1.

{\ }

By Prop. \ref{propConcentrationNearZero}, $\exists \veps > 0$:
$\vect{x} \in \widehat{B}$, $D(\vect{x}) < \veps$ $\Rarr$ $\rho(\vect{x}; \vect{s}) < \frac{1}{3} \delta$.
When $r_k = \frac{r}{2^k} < \frac{1}{2}\veps$, for two balls $B[\vect{c}'; r_k]$, $B[\vect{c}''; r_k]$
to be in $\mathcal{C}_k \subseteq \mathcal{S}$ it is necessary that
$D(\vect{c}'), D(\vect{c}'') \lesq 2 r_k < \veps$, thus
$\rho(\vect{c}'; \vect{c}'') \lesq \rho(\vect{c}'; \vect{s}) + \rho(\vect{s}; \vect{c}'') <
\frac{2}{3} \delta$, --- for big enough $k$ the condition (1) holds.

Obviously, the condition (2) holds when $\frac{r}{2^k} < \frac{1}{3} \delta$
$\LRarr$ $k > \log_2 \frac{3r}{\delta}$.

As soon as the process reaches $k$ such that both conditions hold and halts,
$\vect{z}_k$ is the sought approximation of $\vect{s}$:
suppose $\vect{s} \in B[\vect{c}; r_k] \in \mathcal{C}_k$, then

\cenlin{$\rho(\vect{z}_k; \vect{s}) \lesq \rho(\vect{z}_k; \vect{c}) + \rho(\vect{c}; \vect{s}) <
\frac{2}{3} \delta + \frac{1}{3} \delta = \delta$}

\section{SR in finite-dimensional normed spaces}

\forceparindent
Now we denote by $(X; \| \cdot \|)$ the normed space over the field $\mathbb{R}$ of real numbers.
$\vect{\theta}$ is the zero of $X$ as linear vector space.

We apply the same ``refining cover by defect'' (RCD) method to approximate $\vect{s}$,
only the RC itself becomes more ``constructible''
due to the usage of bases and coordinates.
Most of the reasonings above for metric spaces remain though,
with usual $\rho(\vect{x}; \vect{y}) = \| \vect{x} - \vect{y} \|$.

We keep the constraints $\clubsuit$1--3. As for $\clubsuit$M1, it is provided by

{\ }

$\clubsuit$N1. $X$ is finite-dimensional: $\dim X = m \in \mathbb{N}$.

{\ }

$\bullet$ $X$ is a complete (Banach) space.

$\bullet$ If $A \subseteq X$ is closed and bounded ($A \subseteq B[\vect{\theta}; R]$), then $A$ is compact.

In particular, any closed ball is compact --- $\clubsuit$M1.

$\bullet$ If $L$ is a (linear) subspace of $X$, $L < X$, then $L$ is a closed subspace.

$\bullet$ If $L < X$ and $\vect{x} \notin L$, then
$\rho(\vect{x};L) = \inf\limits_{\vect{u} \in L} \| \vect{x} - \vect{u} \| > 0$
and $\exists \vect{h} \in L$: $\| \vect{x} - \vect{h} \| = \rho(\vect{x}; L)$.

{\ }

In principle, we could take any normalized basis $E = \{ \vect{e}_j \}_{j=1}^m$ of $X$
(that is, $\| \vect{e}_j \| \equiv 1$, $E$ is linearly independent, and
$X = L(E) = \bigl\{ \sum\limits_{j=1}^m x_j \vect{e}_j \mid
x_j \in \mathbb{R}, j=\ovline{1,m} \bigr\}$).
However, for the sake of optimization of the refining cover
we prefer the so-called \textit{Auerbach bases}.

We denote by $X^*$ the \textit{dual}, or \textit{adjoint}, space of $X$, that is,
the space of all linear bounded functionals $f \colon X \rarr \mathbb{R}$.
$\| \cdot \|_*$ is the norm of $X^*$.

$\bullet$ $\forall f \in X^*$, $\forall \vect{x} \in X$: $| f(\vect{x}) | \lesq \| f \|_* \cdot \| \vect{x} \|$.

{\ }

\textbf{Auerbach theorem}
(\cite{auerbach1929}, \cite[20.12]{jameson1974}).
\textit{There exist $\{ \vect{e}_j \}_{j=1}^m \subset X$ and $\{ f_j \}_{j=1}^m \subset X^*$ such that
$\| \vect{e}_j \| = \| f_j \|_* = 1$, $j=\ovline{1,m}$ \emph{(normality)},
and $f_i(\vect{e}_j) = \delta_{ij}$, $i, j = \ovline{1, m}$ \emph{(biorthogonality)}.}

{\ }

Let $\mathbb{E} = \{ \vect{e}_j \}_{j=1}^m$ from Auerbach theorem.
$\forall \vect{e}_j$, $\forall \vect{u} = \sum\limits_{i \ne j} u_i \vect{e}_i \in
L_{-j} = L(\mathbb{E} \backsl \{ \vect{e}_j \})$ we have

\cenlin{$1 = |\delta_{jj} - 0| = |f_j(\vect{e}_j) - \sum\limits_{i\ne j} u_i f_j(\vect{e}_i) | =
	| f_j(\vect{e}_j - \vect{u}) | \lesq \| f_j \|_* \cdot \| \vect{e}_j - \vect{u} \| = \| \vect{e}_j - \vect{u} \|$}

\noindent
thus $\rho (\vect{e}_j; L_{-j}) \greq 1$.
On the other hand, $\vect{\theta} \in L_{-j}$ and $\| \vect{e}_j - \vect{\theta} \| = 1$, so

\hfill $\rho (\vect{e}_j; L_{-j}) = \| \vect{e}_j \| = 1$, $j = \ovline{1, m}$ \hfill \equ{eqOrthonormBasis}

\noindent
--- in addition to being the normalized basis of $X$ ($\mathbb{E}$ is linearly independent
and $|\mathbb{E}| = \dim X$) this Auerbach basis $\mathbb{E}$ has the ``orthogonality'' property.

See also \cite[11.1, p. 517--519]{montesinos2015}.
If we have some non-Auerbach basis $E$ of $X$ and want to ``construct'', or approximate,
the Auerbach one $\mathbb{E}$ (i.e. calculate the coordinates of $\vect{e}_j \in \mathbb{E}$ in $E$)
using the referenced ``canonical'' proof, which involves the maximization of the determinant,
then we can search for that maximum in the $m^2$-dimensional space
of the coordinates of the $m$-tuples of the points on $\{ \vect{x} \in X \colon \| \vect{x} \| = 1 \}$, ---
a complicated task as $m$ increases;
on the other hand, we perform this search only once for given $(X; \| \cdot \|)$.

{\ }

\textbf{Refining cover.}
We describe (or just recall) the cover of $B = B[\vect{\theta}; 1]$
by the ``lattice'' of the closed balls of radius $\frac{1}{2}$; cf. \cite[2.2, 6.3]{fowkes2013}.
Let $\vect{x} \in B$ and $\vect{x} = \sum\limits_{j=1}^m x_j \vect{e}_j$.
For $x_j \ne 0$

\cenlin{$1 \greq \| \vect{x} \| = |x_j| \cdot \| \vect{e}_j + \sum\limits_{i \ne j} \frac{x_i}{x_j} \vect{e}_i \|$}

Since $\bigl( -\sum\limits_{i \ne j} \frac{x_i}{x_j} \vect{e}_i \bigr) \in L_{-j} =
L\bigl(\{ \vect{e}_1; \ldots ; \vect{e}_{j-1}; \vect{e}_{j+1}; \ldots; \vect{e}_m \}\bigr)$, we obtain
$1 \greq |x_j| \cdot \rho(\vect{e}_j; L_{-j})$.

By construction of $\mathbb{E}$, we have \equref{eqOrthonormBasis}: $\rho(\vect{e}_j; L_{-j}) = 1$, therefore $|x_j| \lesq 1$.

Let $c_i = -1 + \frac{i}{m}$, $i = \ovline{0, 2m}$: we break $[-1; 1]$
into the segments $[c_i; c_{i+1}]$ of length $\frac{1}{m}$.

Consider the set of the balls $\mathcal{B} = \bigl\{ B[\vect{c}; \frac{1}{2}] \mid
\vect{c} = \sum\limits_{j=1}^m c_{i_j} \vect{e}_j, \: i_j = \ovline{0, 2m}, \: j=\ovline{1, m} \bigr\}$.
There are $(2m + 1)^m$ of them, and we instantly remove from $\mathcal{B}$
the balls $B' = B[\vect{c}; \frac{1}{2}]$ such that $\| \vect{c} \| > \frac{3}{2}$,
because $B' \cap B = \nullset$ then ($\exists \vect{x} \in B' \cap B$ $\Rarr$
$\| \vect{c} \| \lesq \| \vect{c} - \vect{x} \| + \| \vect{x} \| \lesq \frac{1}{2} + 1$).

We claim that $\forall \vect{x} \in B$ $\exists \vect{c}$:
$\vect{x} \in B[\vect{c}; \frac{1}{2}] \in \mathcal{B}$.
To obtain such $\vect{c}$, we take $c_{i_j}$ that is closest to $x_j$
($i_j = \mathrm{rnd} (m(1 + x_j))$, where
$\mathrm{rnd}(x) = \lfloor x \rfloor + \lfloor 2\{ x \} \rfloor = \lfloor 2x \rfloor - \lfloor x \rfloor$),
then $|x_j - c_{i_j}| \lesq \frac{1}{2m}$ and

\cenlin{$\| \vect{x} - \vect{c} \| \lesq \sum\limits_{j=1}^m | x_j - c_{i_j} | \cdot \| \vect{e}_j \| \lesq
\sum\limits_{j=1}^m \frac{1}{2m} = \frac{1}{2}$} 

Thus $\mathcal{B}$ is the cover of $B$.

Analogously, scaled and translated $\breve{\mathcal{B}} = \vect{z} + r\mathcal{B} =
\bigl\{ B[\vect{z} + \vect{c}; \frac{r}{2}] \colon B[\vect{c}; \frac{1}{2}] \in \mathcal{B} \bigr\}$
is the sought cover of $B[\vect{z}; r]$,
and $\forall B[\vect{z}'; \frac{r}{2}] \in \breve{\mathcal{B}}$: $\| \vect{z}' - \vect{z} \| \lesq \frac{3}{2} r$
($\sim$ Prop. \ref{propRefinCoverMetric}).

\section*{Remarks}

\forceparindent
\textbf{1. Defects.} Other defect functions

\cenlin{$D_1(\vect{x}) = \frac{2}{n^2} \sum\limits_{1 \lesq i < j \lesq n} \bigl| \tau_i(\vect{x}) - \tau_j(\vect{x}) \bigr|$,
$D_2(\vect{x}) = \frac{1}{n} \sum\limits_{i=1}^n \bigl[ \tau_i(\vect{x}) -
\frac{1}{n} \sum\limits_{j=1}^n \tau_j(\vect{x}) \bigr]^2$,}

\cenlin{$D_I(\vect{x}) = I(\vect{x}) = \max\limits_i \tau_i(\vect{x}) - \min\limits_i \tau_i(\vect{x}) =
\max\limits_{1 \lesq i < j \lesq n} \bigl| \tau_i(\vect{x}) - \tau_j(\vect{x}) \bigr|$,}

\cenlin{\textellipsis}

\noindent
have the properties similar to those of $D(\cdot)$. For instance,

\cenlin{$\bigl| D_2(\vect{x}) - D_2(\vect{y}) \bigr| \lesq 8 M \rho(\vect{x}; \vect{y})$}

\noindent
where $M = \max\limits_{i, j} \rho(\vect{r}_i; \vect{r}_j)$.

{\ }

\textbf{2. Issues with gradient method} (GM) of searching for the minima
of the defect function $f_D(\vect{x})$,
which starts at the initial point $\vect{x}_0$ and ``moves'' in the direction
of the \textit{steepest descent} (another name of this method) of $f_D(\vect{x})$;
see \cite[6.6]{byrne2008}, \cite[25]{ostrowski1973}.

Obviously, $\vect{s}$ is a local (and global) minimum, locmin, of $f_D(\cdot)$.
However, in general case there can be more than one locmin,
even without disturbances caused by noise
(cf. \cite[5.2]{compagnoni2017}, \cite[9.4]{huang2004}):

\begin{example}\label{exmpOtherLocMin}
Consider the defect $D_2(\cdot)$,
which is also the variance of the random variable with equiprobable values $\tau_i(\vect{x})$.
Let $(X; \rho) = \mathbb{R}^2_2$, $\vect{s} = (0; 0)$, and the sensors
	
\cenlin{$\vect{r}_1 = (8; 6)$, $\vect{r}_2 = (5; 5)$, $\vect{r}_3 = (-2; 6)$,
$\vect{r}_4 = (-6; 4)$, $\vect{r}_5 = (-10; 2)$}
\end{example}

\begin{exampledescript}
``Numerical experiments'' show that $D_2(\vect{x})$ has locmin at
$\vect{b} \approx (-3.6901; 21.5627)$, at that $D_2(\vect{b}) \approx 0.69044$.
Therefore we cannot start GM at arbitrary initial point
to search for the solution ($\vect{r}_5 = 2\vect{r}_4 - \vect{r}_3$,
so \cite[Prop. 4]{goncharov2018} implies the uniqueness of the solution $\vect{s}$).
\end{exampledescript}

Similar configurations exist for higher dimensionalities.

{\ }

Moreover, GM that starts at the sensor nearest to $\vect{s}$ can
converge to the locmin $\vect{b} \ne \vect{s}$:

\begin{example}\label{exmpOtherLocMinFromNearest}
Let $(X; \rho) = \mathbb{R}^2_2$, $\vect{s} = (0; 0)$, and the sensors
	
\cenlin{$\vect{r}_1(1.885; 0.014)$, $\vect{r}_2(2.523; -0.76)$, $\vect{r}_3(2.552; -0.756)$,
$\vect{r}_4(2.94; -0.78)$, $\vect{r}_5(2.081; 0.986)$}
\end{example}

\begin{exampledescript}
GM with the initial point $\vect{r}_1$, which is nearest to $\vect{s}$,
converges to the locmin of $D_2(\vect{x})$ at
$\vect{b} \approx (2.039; 0.253)$, $D_2(\vect{b}) \approx 0.00318$.
\end{exampledescript}

Again, similar behaviour can occur in $\mathbb{R}^m_2$ for $m > 2$.

{\ }

\textbf{3. Towards Noise.} When, instead of exact $t_i$,
we know only ``shifted'' $\widehat{t}_i = t_i + \xi_i$
and $\widehat{\tau}_i (\vect{x}) = \widehat{t}_i - \rho(\vect{x}; \vect{r}_i)$
(noises $\xi_i$ are supposed to be random variables with certain properties),
$\widehat{D} = \frac{1}{n^2} \sum\limits_i \bigl| \sum\limits_j \bigl[ \widehat{\tau}_i - \widehat{\tau}_j \bigr] \bigr|$
is ``distorted''; we lose Prop. \ref{propDefectZeroSol}, Prop. \ref{propConcentrationNearZero},
and what is built on top of them.
Maybe $\widehat{D}$ has a continuum of zeros, maybe none.

If we opt to keep using the root finding approach rather than the optimization one,
then the following crude, vaguely described ``trick'' may be applied,
assuming $|\xi_i| \lesq \gamma$ for small enough $\gamma$: consider
$\wtilde{D} = |\widehat{D} - 2 \gamma|$.

It is easy to see that $|\widehat{D} - D| \lesq 2\gamma$,
thus $\widehat{D}(\vect{s}) - 2\gamma \lesq 0$, while at some distant $\vect{x}$
presumably $\widehat{D}(\vect{x}) - 2\gamma > 0$; $\wtilde{D}$ has zeros.
Like $\widehat{D}$, $\wtilde{D}$ is a ``distortion'' of $D$,
but due to $\gamma \approx 0$ this distortion should be small enough
for the (continuum of) zeros of $\wtilde{D}$ to be ``not too far'' from $\vect{s}$.
These zeros form the closed (topologically and geometrically)
``surface(s)'' $Z = \{ \vect{x} \in X \mid \wtilde{D}(\vect{x}) = 0 \}$ around or near $\vect{s}$;
in a sense, the distortion ``inflates'' single zero, turning it into surface(s).

Prop. \ref{propDefectDiffUppEstim} remains valid for $\wtilde{D}$,
and instead of Prop. \ref{propConcentrationNearZero} we'll have its analogue
with $\rho(\vect{x}; \vect{s})$ replaced by $\rho(\vect{x};Z)$.
We use the test ``$2r_k < \wtilde{D}(\vect{c})$'', and the refining cover
$\{ \mathcal{C}_k \}_{k=0}^{\infty}$ constisting of $B[\vect{c}; r_k]$
that do not satisfy this inequality ``converges'' to $Z$,
which stays mostly within the union of the balls from $\mathcal{C}_k$.
At some iteration we halt and take some point ``between'' the centers of these balls
(their mean in normed space, for instance)\textellipsis{} and rely on this point being close enough to $\vect{s}$.

The locmins cause one of evident drawbacks of this trick:
false, or ``ghost'', solutions can appear near such minima,
relatively far from the true source $\vect{s}$.
At least they should not appear,
and $\mathcal{C}_k$ shouldn't break into the disjointed groups of the balls as $k \rarr \infty$,
if $2\gamma < \mu$, where $\mu$ is the minimal value of $\widehat{D}(\vect{b})$
at the locmins $\vect{b}$ that are not the ``descendants'' of $\vect{s}$.

Meta-refinement: we run, in parallel,
several instances of $|\widehat{D} - \lambda|$-trick with different $\lambda$
(e.g. $\lambda_{ij} = \pm \frac{i}{2^{j-1}} \gamma)$,
compare how the respective covers behave, and spawn new instances if needed.

{\ }

\textbf{Le encouragement.}
From H. Lebesgue's letter to E. Borel about the geometric approximations in sound ranging,
Feb (?) 1915 (original text at \cite[p. 323]{lebesgue1991}, translation at \cite[p. 146--147]{aubin2014}):

\cenlin{\textsl{{\small\begin{tabular}{p{7cm}|p{6.5cm}}
``\textellipsis{} Au fond la chose ne m'int\'eresse plus:\newline
1$^\circ$ parce que j'ai constat\'e que je ne vois dans les lunettes que les objets brillamment \'eclair\'es;\newline
2$^\circ$ parce que \textellipsis{}''
&
``\textellipsis{} At bottom, I have lost interest in this:\newline
1$^\circ$ because I realized that I only see in the telescopes objects that are brightly lit;\newline
2$^\circ$ because \textellipsis{}''
\end{tabular}}}}

\end{document}